%% file: causality.tex
\documentclass[12pt,reqno]{amsart}
\usepackage{etex}
\usepackage{amsmath,amsthm,amsfonts,amssymb,amscd,flafter,pinlabel, booktabs}
\usepackage[mathscr]{eucal}
\usepackage{graphics}
 \usepackage[all,knot,arc]{xy}

\usepackage[usenames,dvipsnames]{color}
\usepackage{subfigure}
\usepackage{graphicx}
 \usepackage{tabu}
 \usepackage{epstopdf}
 \usepackage[T1]{fontenc}
 \usepackage{scalefnt}

\usepackage{bbm}
\usepackage{mathtools}

\usepackage{pgf}

\usepackage{tikz}
\usetikzlibrary{arrows,automata}
\usetikzlibrary{matrix,arrows}
\tikzset{cdlabel/.style={above,sloped,%
    execute at begin node=$\scriptstyle,execute at end node=$}}
\tikzset{algarrow/.style={->, thick}}   
\tikzset{alb/.style={->, bend right=55, thick}}
\tikzset{arb/.style={->, bend left=25, thick}} 
\tikzset{al/.style={->, bend right=20, thick}}
\tikzset{ar/.style={->, bend left=20, thick}}
\tikzset{unst1/.style={->, dashed, bend right=30, thick}}
\tikzset{unst2/.style={->, dashed, bend left=30, thick}}
\tikzset{als/.style={->, bend right=15, thick}}
\tikzset{ars/.style={->, bend left=15, thick}}
\tikzset{blgarrow/.style={->, thick}}
\tikzset{clgarrow/.style={->, thick}}
\tikzset{tensoralgarrow/.style={double, double equal sign distance, -implies}}
\tikzset{tensorblgarrow/.style={double, double equal sign distance, -implies}}
\tikzset{tensorclgarrow/.style={double, double equal sign distance, -implies}}
\tikzset{tensorelgarrow/.style={double, double equal sign distance, -implies}}
\tikzset{modarrow/.style={->, dashed}}
\tikzset{Amodar/.style={->, dashed}}
\tikzset{Dmodar/.style={->, dashed}}
\tikzset{DAmodar/.style={->, dashed}}

\usepackage[latin1]{inputenc}

\usepackage[colorlinks=true,linkcolor=blue,citecolor=blue,urlcolor=blue]{hyperref}
\usepackage{xcolor}

\tikzset{cdlabel/.style={above,sloped,
    execute at begin node=$\scriptstyle,execute at end node=$}}    
\tikzset{al/.style={->, bend right=45, thick}}
\tikzset{ar/.style={->, bend left=45, thick}}

\normalfont\upshape

\usepackage[lmargin=1in,rmargin=1in,tmargin=1in,bmargin=1in]{geometry}

\newtheorem{theorem}{Theorem}

\theoremstyle{definition}
\newtheorem*{remark}{Remark}

\newcommand{\R}{\ensuremath{\mathbb{R}}}
\newcommand{\Z}{\ensuremath{\mathbb{Z}}}


\begin{document}

\title{Khovanov homology and causality in spacetimes}

\author{Vladimir Chernov}
\address {Department of Mathematics, Dartmouth College\\ Hanover, NH 03755}
\email {vladimir.chernov@dartmouth.edu}
\urladdr{\href{http://math.dartmouth.edu/~chernov}{http://math.dartmouth.edu/~chernov}}
\author{Gage Martin}
\address {Department of Mathematics, Boston College\\ Chestnut Hill, MA 02467}
\email {martaic@bc.edu}
\author{Ina Petkova}
\address {Department of Mathematics, Dartmouth College\\ Hanover, NH 03755}
\email {ina.petkova@dartmouth.edu}
\urladdr{\href{http://math.dartmouth.edu/~ina}{http://math.dartmouth.edu/~ina}}



\begin{abstract}
We observe that Khovanov homology detects causality in  $(2+1)$-dimensional globally hyperbolic spacetimes whose Cauchy surface is homeomorphic to $\R^2$.
\end{abstract}

\maketitle

A spacetime $X$ is a time-oriented Lorentz manifold. Here a Lorentz metric is assumed to have signature $(+, \cdots, +, -)$. The spacetime is \emph{globally hyperbolic} if it has a Cauchy surface, i.e.~a subset $\Sigma$ such that each inextensible curve $\gamma$ with $\gamma'(t)\cdot \gamma'(t)\leq 0$ intersects it exactly once~\cite[pp.~211-212]{he}. The Cauchy surface can be taken to be smooth and spacelike, i.e.~such that the restriction of the Lorentz metric to it is Riemann. The globally hyperbolic spacetime is then diffeomorphic to $\Sigma \times \mathbb R$~\cite{bs3,bs2,bs1}.

Globally hyperbolic spacetimes are arguably the most important and studied class of spacetimes. One of the versions of the Strong Cosmic Censorship Conjecture of Penrose~\cite{penrose} is that all physically relevant spacetimes are like this. 

Two points (events) $x,y\in X$ are \emph{causally related} if there is a curve $\gamma$ connecting the two points with $\gamma'(t)\cdot \gamma'(t)\leq 0$, i.e.~if one can get from one point to the other without exceeding the light speed.   
The space $N_X$ of unparameterized  future-directed light rays in $X$ can be identified with the total space $ST^*\Sigma$ of the spherical cotangent bundle of  a Cauchy surface $\Sigma$. The set of lights rays through a point $x$ is then identified with the sphere $S_x\subset N_X$ called the \emph{sky of $x$}.

Assuming that the Cauchy surface $\Sigma$ is not homeomorphic  to $\R P^2$ or $S^2$, S.~Nemirovski and the first author proved the Low Conjecture~\cite{low-phd, low1, low2, LowLeg, LowNull}, which states that  events $x,y$ in a $(2+1)$-dimensional globally hyperbolic spacetime $X$ (with $\Sigma \neq S^2, \mathbb R P^2$) are causally related if and only if their skies $S_x, S_y$ are linked in $N_X$, see \cite{cn}. 
Here \emph{linked} means that  the pair is {\bf not} isotopic to a pair of fibers of the $S^1$-bundle $ST^*\Sigma\to \Sigma$ or that $S_x, S_y$ intersect, meaning that $x,y$ are on a common light ray. One can show that this does not depend on the choice of a Cauchy surface $\Sigma$ and hence on the identification $N_X=ST^*\Sigma.$

For the case where the Cauchy surface of $X$ is homeomorphic to $\mathbb R^2$, these skies are circles in the solid torus $N_X=S^1\times \mathbb R^2$, each isotopic to the longitude of the solid torus. 
In \cite{nat-tod} Nat\'ario and Tod used the Kauffman polynomial (which is related to the Jones polynomial by a variable change) of the pair of skies to study 
causality. They computed the polynomial for a big class of examples of pairs of skies of causally related events. The polynomial was nontrivial, thus showing the skies were indeed linked. This provided numerical evidence supporting the Low Conjecture, which at the time was unknown to be true. Completely detecting causality with this method is unlikely, however. There are infinitely many examples of pairs of linked longitudes in $S^1\times \R^2$ which become unlinked in $S^3\supset S^1\times \R^2$, see Figure~\ref{fig:link-in-t} for an example, so one would at least need to refine the computation to keep track of the solid torus. Further, it is unknown whether the Kauffman polynomial would detect the connected sum of two Hopf links in $\mathbb R^3$ (i.e.~ the trivial pair of longitudes in the solid torus).

Khovanov homology assigns to a link $L \subset S^3$ a bigraded $\Z/2$-module $\mathit{Kh}(L; \Z/2)$, whose Euler characteristic is the Jones polynomial. It is known that Khovanov homology carries strictly more information than the Jones polynomial. A variant of Khovanov homology called annular Khovanov homology assigns to an annular link $L \subseteq A \times [0,1]$ a triply graded $\Z/2$-module $\mathit{AKh}(L; \Z/2)$. Combining recent detection results about Khovanov homology and annular Khovanov homology~\cite{grigsby_sutured_2014,john_a_baldwin_categorified_2015,min-kh,xie_instanton_2019} with the results in~\cite{cn}, we get the following.

\begin{theorem}
Khovanov homology and annular Khovanov homology both detect causal relation for events in $(2+1)$-dimensional globally hyperbolic spacetimes with a Cauchy surface homeomorphic to $\mathbb R^2$.
\end{theorem}

\begin{proof}
If the skies $S_x, S_y$ intersect, then the two events $x,y$ belong to a common light ray, and hence are causally related. So below we assume that $S_x\cap S_y=\emptyset$.

Let $V= S^1\times \R^2$ be a solid torus and let $\lambda = S^1\times \{p\}$ be a longitude of $V$. 

First we show how existing results imply the theorem for annular Khovanov homology. Since $V=S^1\times \mathbb R^2=(S^1\times \mathbb R)\times \mathbb R$, we can regard $V$ as the thickened annulus $A \times I=A\times \mathbb R$.  Let $L= K_1\sqcup K_2$ be a $2$-component link in  $V$, such that each link component $K_i$ is isotopic to $\lambda$.  Determining whether $L$ is unlinked (isotopic to a pair of fibers) in $V$ is equivalent to determining whether $L$ is the (closure of the)  trivial two-braid in $A \times I$.

By~\cite[Corollary~1.2]{grigsby_sutured_2014} and~\cite[Corollary~1.4]{xie_instanton_2019}, annular Khovanov homology is known to distinguish  braids in $A \times I$ from other annular links. Additionally, by~\cite[Theorem~3.1]{john_a_baldwin_categorified_2015}, annular Khovanov homology is known to distinguish the trivial braid closure from other braid closures. In particular, $x,y\in X$ are causally unrelated if and only if $\mathit{AKh}(S_x \sqcup S_y; \Z/2)\cong \mathit{AKh}(U_2; \Z/2)$. Here $U_2$ is the closure of the trivial two-braid in $A \times I$.

Now we address Khovanov homology.  Let $L$ be an $n$-component link in $V$. Using the standard embedding of  $V$ into $S^3$ (so that $\lambda$ bounds a disk disjoint from some other $S^1$ fiber of $V$), we can equivalently think of the pair $(V, L)$ as the pair $(S^3, L')$, where $L'$ is the $(n+1)$-component link consisting of $L$, together with the meridian $\mu$ of $V$. See Figure~\ref{fig:link-in-t}. 

 Let $L= K_1\sqcup K_2$ be a $2$-component link in  $V$, such that each link component $K_i$ is isotopic to $\lambda$.  Determining whether $L$ is unlinked (not isotopic to a pair of fibers) in $V$ is equivalent to determining whether $L'$ is the connected sum of two Hopf links.  

\begin{figure}[h]
  \centering
  \labellist
  \pinlabel  {\textcolor{red}{$\mu$}} at -10 160
    \pinlabel  {\textcolor{blue}{$\lambda$}} at 9 71
  \endlabellist
  \includegraphics[scale=.45]{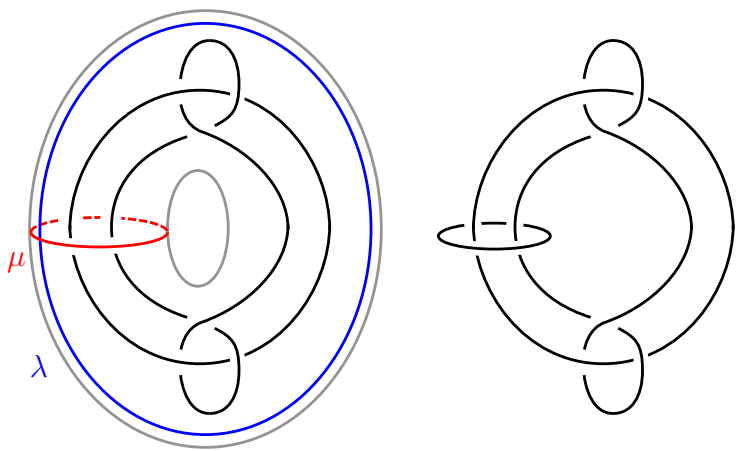}
  \caption{Left: a $2$-component link in the solid torus $V = S^1\times \R^2$. The fiber longitude $\lambda$ is drawn in blue, and the meridian $\mu$ is in red. Right: The corresponding $3$-component link $L'$.}
  \label{fig:link-in-t}
\end{figure}
 By \cite[Corollary~1.3]{min-kh}, Khovanov homology detects the connected sum of two Hopf links among all $3$-component links. In particular, $x,y\in X$ are causally unrelated if and only if $\mathit{Kh}(S_x\sqcup S_y\sqcup \mu; \Z/2)\cong \mathit{Kh}(P_3; \Z/2)$. Here $P_3$ is the standard connected sum of two Hopf links. 
\end{proof}

\begin{remark}
The universal cover $\widetilde X$  of a globally hyperbolic $(2+1)$-dimensional spacetime $X$ with a Cauchy surface $\Sigma$ is a globally hyperbolic spacetime whose Cauchy surface $\widetilde \Sigma$ is the universal cover of $\Sigma$~\cite{ChernovRudyak}. If $x,y\in X$ are two causally related points, then there is a path $\gamma$ with $\gamma'(t)\cdot \gamma'(t)\leq 0$ connecting them. Hence, its lift $\tilde \gamma$ connects some lifts $\tilde x, \tilde y$ of $x,y$ to $\widetilde X$ which are then causally related in $\widetilde X$. (Of course, if any two points in $\widetilde X$ are causally related, then their projections to $X$ are also causally related.)

When the universal covering of $\Sigma$ is not $S^2$, it has to be homeomorphic to $\mathbb R^2$, and we can apply the covering construction from above to show that Khovanov homology detects causality for all globally $(2+1)$-dimensional spacetimes whose Cauchy surface is not homeomorphic to $S^2$ or $\mathbb R P ^2$. When $\Sigma$ itself is not homeomorphic to $\R^2$ however, this would  require one to check that infinitely many pairs of possible lifts of skies are unlinked, or vice versa that there is a pair of lifts whose skies are linked.
\end{remark}

\begin{remark}
For higher dimensional globally hyperbolic spacetimes, causality can be interpreted as a Legendrian linking of a pair of skies, see for example~\cite{cn, cn2, cn3, c}. However, to detect the Legendrian unlink X
one would have to develop and use techniques other than Khovanov homology, which is 
$3$-dimensional and topological (rather than contact) in nature.
\end{remark}

\begin{remark}
Similar arguments to those contained in~\cite{grigsby_sutured_2014} show that link Floer homology detects braids. This braid detection result is known to some experts but is missing from the literature. A version of the argument  was communicated to the second author by John Baldwin~\cite{BaldwinPersonalCommunication}. A proof will appear in an upcoming paper by the second author.  
Combined with a result from~\cite{john_a_baldwin_categorified_2015} that link Floer homology detects the trivial braid, this implies that link Floer homology also detects causality in this setting.
\end{remark}

\subsection*{Acknowledgments}
We thank Robert Low for a valuable discussion during his recent visit to Dartmouth College. This work was partially supported by a grant from the Simons Foundation
(\# 513272 to Vladimir Chernov). Ina Petkova received support from NSF Grant DMS-1711100.

\bibliographystyle{alpha}

\bibliography{master}

\end{document}